\documentclass[12pt, notitlepage]{amsart}
\usepackage{latexsym, amsfonts, amsmath, amssymb, amsthm, cite, tabls, paralist}

\newtheorem{theorem}{Theorem}[section]
\newtheorem{lemma}[theorem]{Lemma}
\newtheorem{proposition}[theorem]{Proposition}

\newtheorem*{nonumtheorem}{Theorem}
\newtheorem{corollary}[theorem]{Corollary}

\theoremstyle{remark}
\newtheorem{remark}[theorem]{Remark}

\theoremstyle{definition}
\newtheorem{definition}[theorem]{Definition}

\theoremstyle{example}

\numberwithin{equation}{section}

\newcommand{\Z}{\mathbb{Z}}
\newcommand{\R}{\mathbb{R}}

\newcommand{\C}{\mathbb{C}}
\newcommand{\E}{\mathbb{E}}

\usepackage{graphicx}
\usepackage{mathrsfs}

\usepackage{varioref}

\begin{document}

\title{Finding Linear Patterns of Complexity One}
\author{Xuancheng Shao}
\address{Department of Mathematics\\ Stanford University\\ 450 Serra Mall, Bldg. 380\\ Stanford, CA 94305-2125}

\begin{abstract}
We study the following generalization of Roth's theorem for $3$-term arithmetic progressions. For $s\geq 2$, define a nontrivial $s$-configuration to be a set of $s(s+1)/2$ integers consisting of $s$ distinct integers $x_1,\cdots,x_s$ as well as the averages $(x_i+x_j)/2$ ($1\leq i<j\leq s$). Our main result states that if a set $A\subset [N]$ has density $\delta\gg (\log N)^{-c(s)}$ for some positive constant $c(s)>0$ depending on $s$, then $A$ contains a nontrivial $s$-configuration. This improves on the previous bound of the form $\delta\gg (\log\log N)^{-c(s)}$ due to Dousse \cite{dousse}. We also deduce, as a corollary, an improvement of a problem involving sumfree subsets.
\end{abstract}

\maketitle

\section{Introduction}

The celebrated Roth's theorem \cite{roth} states that every dense subset of the integers contains a nontrivial $3$-term arithmetic progression ($3$-AP).

\begin{nonumtheorem}[Roth]
Let $0<\delta<1$ be a positive real. For a sufficiently large positive
integer $N$, any subset $A\subset [N]$ with $|A|\geq\delta N$ must
contain a nontrivial $3$-term arithmetic progression.
\end{nonumtheorem}

Here we use $[N]$ to denote the set of positive integers up to $N$. In fact, Roth obtained the more precise bound $\delta\gg (\log\log
N)^{-1}$. This has been subsequently improved by Szemer\'{e}di \cite{szemeredi} and
Heath-Brown \cite{heath-brown} to $\delta\gg (\log N)^{-c}$ for some small positive
constant $c$. Later Bourgain \cite{bourgain1,bourgain2} improved this to $\delta\gg (\log
N)^{-2/3+o(1)}$. Currently the best quantitative bound is due to Sanders \cite{sanders}:
$\delta\gg (\log N)^{-1+o(1)}$. On the other hand, a construction
of Behrend shows that there exists a subset without nontrivial
$3$-APs with density about $\exp(-C\sqrt{\log N})$ for some constant $C>0$.

Roth's argument relies on a dichotomy which either guarantees the existence of a nontrivial $3$-AP, or leads to a density increment in a shorter subprogression.  See \cite{tao-vu} for an exposition of Roth's density increment argument as well as some of the later improvements.

We are interested in the following generalization of $3$-APs. Fix a
positive integer $s\geq 2$. An $s$-configuration is set of
$s(s+1)/2$  integers
\begin{equation}\label{config}
\{n_i+n_j+a:1\leq i\leq j\leq s\},
\end{equation}
where $n_1,\cdots,n_s,a\in\Z$. For example, when $s=2$, the
three integers $(2n_1+a,n_1+n_2+a,2n_2+a)$ form a $3$-AP. An $s$-configuration is said to be nontrivial if $n_1,\cdots,n_s$ are all
distinct. A geometrical way to think about $s$-configurations is
that they contain $s$ points on the real line and all midpoints
between any two of these points.

It is natural to ask whether every dense subset of the integers contains a nontrivial $s$-configuration (for fixed $s$). This is proved by Dousse \cite{dousse} using Roth's density increment argument.

\begin{theorem}\label{thm1}
Let $A\subset [N]$ with $|A|=\delta N$. Suppose that
\[ \delta\gg (\log\log N)^{-\frac{1}{s(s+1)-1}}. \]
Then $A$ contains a nontrivial $s$-configuration.
\end{theorem}

A crucial reason why Roth's argument applies to $s$-configurations is that the set of linear forms in \eqref{config} has {\em complexity} $1$ (regardless of the value of $s$), and thus controlled by Gowers $U^2$-norm. The Gowers uniformity norms are the main tools in Gowers' proof \cite{gowers1,gowers2} of Szemer\'{e}di's theorem. The notion of complexity of linear forms seems to originate from \cite{complexity}; for the precise definition and related results see \cite{higher}.

For comparison, the set of linear forms corresponding to $(k+1)$-APs has complexity $k-1$, and thus controlled by Gowers $U^k$-norm. While analyzing the $U^2$-norm is more or less equivalent to doing Fourier analysis, dealing with the $U^k$-norm when $k>2$ becomes significantly more difficult. Indeed this is the main topic of the book \cite{higher}. Fortunately for us, to study $s$-configurations it is sufficient to consider the $U^2$-norm.

We now state our main result:

\begin{theorem}\label{thm2}
Let $A\subset [N]$ with $|A|=\delta N$. Suppose that
\[ \delta\geq 100\left(\frac{\log\log N}{\log N}\right)^{1/6s(s+1)}. \]
Then $A$ contains a nontrivial $s$-configuration.
\end{theorem}

The exponent $1/6s(s+1)$ can certainly be improved slightly by a
more careful analysis. We shall not do so here. Our argument should
work for other linear patterns of complexity $1$ (and thus the title
for this paper). However doing so in this generality involves some
technical complications. We will thus focus solely on
$s$-configurations.

The proof of Theorem \ref{thm2} relies on Fourier analysis localized
at certain approximate subgroups called Bohr sets. This technique is
originated by Bourgain \cite{bourgain1} in his work on finding 3-APs
and developed further in \cite{localfourier1,localfourier2}. In
order for the local Fourier analysis argument to work on
$s$-configurations, we need a local analogue of Gowers $U^2$-norm;
this is the main innovation of this paper.

It is worth mentioning that a weaker bound of the form $\delta\gg\exp(-c(s)\sqrt{\log\log N})$ can be obtained by following the argument of Green-Tao \cite{green-tao1}, where the problem of finding $4$-APs is studied. The main idea behind their argument is in turn borrowed from Szemer\'{e}di and Heath-Brown's work in the case of $3$-APs. Our problem with $s$-configurations that are governed by Gowers $U^2$-norm is much simpler than the problem with $4$-APs that are governed by Gowers $U^3$-norm. It is thus not surprising that we are able to get a better bound for $s$-configurations. In the ongoing work of Green-Tao they show that a bound of the same quality can be obtained in the case of $4$-APs as well.

The improved bound in Theorem \ref{thm2}, combined with the argument of Sudakov, Szemer\'{e}di, and Vu \cite{sumfree}, leads to an improvement in a problem involving sumfree sets.

\begin{corollary}\label{sumfree}
Any set $A\subset\Z$ of size $n$ contains a subset $B\subset A$ with
size $|B|\geq (\log n)(\log\log n)^{1/2-o(1)}$ such that $B$ is
sumfree with respect to $A$. In other words, $b+b'\notin A$ for any
two distinct elements $b,b'\in B$.
\end{corollary}

Here $o(1)$ denotes a quantity that goes to zero as $n\rightarrow\infty$. For a more detailed discussion on this problem we refer the reader to \cite{sumfree}, where the first superlogarithmic bound on $|B|$ is obtained. The bound $|B|\geq (\log n)(\log\log\log
n)^c$ is obtained by Dousse \cite{dousse} using Theorem \ref{thm1}. In exactly the same way, Theorem
\ref{thm2} immediately gives the bound $|B|\geq (\log n)(\log\log
n)^c$ for some small positive constant $c>0$. The improvement of the exponent $c$ is a consequence of using the Ruzsa embedding lemma instead of Freiman's
theorem to extract a dense subset of a long interval from a set with small doubling.

The rest of the paper is organized as follows. In Section
\ref{sec:bohr} we collect basic properties of Bohr sets, a
multi-dimensional analogue of arithmetic progressions. Theorem
\ref{thm2} is then proved in Sections \ref{sec:u2loc} to
\ref{sec:proof2}. Section \ref{sec:u2loc} develops the theory of a
local analogue of Gowers $U^2$-norm; this captures the intuition of
performing Fourier analysis on Bohr sets (approximate groups)
instead of genuine groups. In Section \ref{sec:gvn-loc} we establish
a dichotomy which either guarantees the existence of a nontrivial
$s$-configuration or leads to a large local $U^2$-norm. Finally in
Section \ref{sec:proof2} we obtain a density increment from the
large local $U^2$-norm and complete the proof of Theorem \ref{thm2}.
The last section contains the proof of Corollary \ref{sumfree}.

\smallskip

{\bf Acknowledgements.} The author is grateful to Ben Green and
Terry Tao for some helpful discussions.

\section{Bohr Sets: Basic Properties}\label{sec:bohr}

In this section we record some basic properties of Bohr sets, all of
which can be found in standard texts in additive combinatorics such
as \cite{tao-vu}.

\begin{definition}[Bohr sets]\label{def:bohr}
Let $\theta=(\theta_1,\cdots,\theta_d)\in\R^d$,  $0<\epsilon<1/2$,
and $M\geq 1$. Define the Bohr set
$\Lambda=\Lambda_{\theta,\epsilon,M}$ to be the set of all $n\in\Z$
with $|n|\leq M$ and $\|n\theta_j\|\leq\epsilon$ for each $1\leq
j\leq d$ (here $\|x\|$ denotes the distance from $x$ to its nearest integer). The positive integer $d$ is called the dimension of
$\Lambda$. For any real number $c>0$, we write $c\Lambda$ for the
dilated Bohr set $c\Lambda=\Lambda_{\theta,c\epsilon,cM}$.
\end{definition}

For example, the interval $[-M,M]$ is the $1$-dimensional Bohr set
$\Lambda_{1,1/2,M}$. A simple heuristic argument shows that an
integer $n\in [-M,M]$ lies in a $d$-dimensional Bohr set
$\Lambda_{\theta,\epsilon,M}$ with probability $\epsilon^d$, and
thus $\Lambda_{\theta,\epsilon,M}$ has expected size roughly
$\epsilon^dM$. This is indeed correct as a lower bound.

\begin{lemma}[Size bound]\label{bohrlow}
Let $\Lambda=\Lambda_{\theta,\epsilon,M}$ be a Bohr set of dimension
$d$. Then $|\Lambda|\geq\epsilon^dM$.
\end{lemma}

Bohr sets are approximate groups in the sense that
$|\Lambda+\Lambda|\approx 2^d|\Lambda|$. However, as $d$ gets large
the doubling also gets large. To get around this unpleasant
behavior, we will work with pairs of Bohr sets $(\Lambda,c\Lambda)$
for some small constant $c\in (0,1)$. Heuristically we expect that
$|\Lambda+c\Lambda|\approx |\Lambda|$. This leads to the following
definition.

\begin{definition}[Regular Bohr sets]\label{def:reg}
A Bohr set $\Lambda=\Lambda_{\theta,\epsilon,M}$ of dimension $d$ is
said to be regular if
\[ 1-100d|c|\leq\frac{|(1+c)\Lambda|}{|\Lambda|}\leq 1+100d|c| \]
whenever $|c|\leq 1/100d$.
\end{definition}

Not all Bohr sets are regular. For example, consider
$\Lambda=\Lambda_{1/2,0.499,M}$, and observe that while $\Lambda$
contains only even integers in $[-M,M]$, the dilated Bohr set $1.01\Lambda$ contains all
integers in $[-M,M]$. On the other hand, it is not hard to see that
$\Lambda_{1/2,\epsilon,M}$ is regular when $\epsilon$ is not too close to
$0.5$. In general, the following lemma shows that regular Bohr sets
exist in abundance.

\begin{lemma}[Finding regular Bohr sets]\label{findreg}
For any Bohr set $\Lambda$, there exists $\alpha\in [1/2,1]$ such
that $\alpha\Lambda$ is regular.
\end{lemma}

Suppose that $\Lambda$ is a regular Bohr set of dimension $d$, and
$\Lambda'=c\Lambda$ for some small constant $0<c<1/100d$. As mentioned above, the approximate
group structure of the pair $(\Lambda,\Lambda')$ is crucial in our argument.
For example, let $f:\Z\rightarrow\C$ be a $1$-bounded function (meaning that $|f|\leq 1$), and
consider the average
\[ \E_{n\in n'+\Lambda}f(n) \]
for some $n'\in\Lambda'$. The regularity of $\Lambda$ allows us to
replace the range $n\in n'+\Lambda$ by $n\in\Lambda$, at the cost of
a small error:
\[ \E_{n\in n'+\Lambda}f(n)=\E_{n\in\Lambda}f(n)+O(cd). \]
We will use this type of estimate again and again without
explicitly mentioning it.


\section{The Local Gowers $U^2$-norm}\label{sec:u2loc}

In this section, we define a local Gowers $U^2$-norm and prove an
inverse theorem for it. To motivate the definition, first recall the
usual $U^2$-norm, defined for a $1$-bounded function
$f:G\rightarrow\C$ on a finite abelian group $G$:
\[
\|f\|_{U^2}^4=\E_{a,d_1,d_2\in
G}f(a)\overline{f(a+d_1)f(a+d_2)}f(a+d_1+d_2).
\]
Geometrically $\|f\|_{U^2}$ counts parallelograms weighted by $f$,
with the parameter $a$ representing one of the vertices of the
parallelogram, and $d_1,d_2$ representing the two edges of it.

\begin{definition}[Local Gowers $U^2$-norm]
Let $\Lambda,\Lambda_1=c_1\Lambda,\Lambda_2=c_2\Lambda_1$ be regular
Bohr sets of dimension $d$, where $0<c_1,c_2<1/100d$. For a
$1$-bounded function $f:\Z\rightarrow\C$, its local Gowers
$U^2$-norm with respect to $(\Lambda,\Lambda_1,\Lambda_2)$ is
defined by
\begin{equation}\label{u2loc}
\|f\|_{U^2(\Lambda,\Lambda_1,\Lambda_2)}^4=\E_{a\in
\Lambda}\E_{n_1,n_1'\in \Lambda_1}\E_{n_2,n_2'\in
\Lambda_2}f(a+n_1+n_2)\overline{f(a+n_1+n_2')f(a+n_1'+n_2)}f(a+n_1'+n_2').
\end{equation}
\end{definition}

Roughly speaking, the local Gowers $U^2$-norm counts those parallelograms with
one side restricted to a narrower Bohr set $\Lambda_1$, and the other side
restricted to an even narrower Bohr set $\Lambda_2$.

Clearly the local Gowers $U^2$-norm is always bounded by $1$. We now
turn to proving an inverse theorem, which says that if the local
Gowers $U^2$-norm of a function $f$ is large, then $f$ must possess
some structure.

\begin{theorem}[Inverse theorem for the $U^2$-norm, local
version]\label{u2inverse-loc} Let $0<\eta<1$ be a parameter. Let
$\Lambda$, $\Lambda_1=c_1\Lambda_1$, $\Lambda_2=c_2\Lambda_1$ be
regular Bohr sets of dimension $d$, where $0<c_1\leq\eta^8/5000d$
and $0<c_2\leq\eta^2/400d$. If
$\|f\|_{U^2(\Lambda,\Lambda_1,\Lambda_2)}\geq\eta$, then
\[
\E_{a\in\Lambda}\sup_{y\in\R}\left|\E_{n_2\in\Lambda_2}f(a+n_2)e(n_2y)\right|^2\geq\frac{\eta^8}{40}.
\]
\end{theorem}

When $\Lambda,\Lambda_1,\Lambda_2$ are replaced by a finite abelian
group $G$, we recover the following inverse theorem for the usual
Gowers $U^2$-norm: if $\|f\|_{U^2}\geq\eta$ for a $1$-bounded
function $f:G\rightarrow\C$, then $|\hat{f}(\gamma)|\gg\eta^4$ for
some character $\gamma\in\hat{G}$. This is slightly weaker than the
usual inverse theorem which says that $|\hat{f}(\gamma)|\gg\eta^2$.
An improvement in the exponent will lead to an improvement in the
exponent in the final bound in Theorem \ref{thm2} as well. This is, however, not our major concern.

\begin{proof}
By the definition \eqref{u2loc}, for at least $\eta^4|\Lambda|/2$
values of $a\in\Lambda$ we have
\begin{equation}\label{largeu2}
\E_{n_1,n_1'\in\Lambda_1}\E_{n_2,n_2'\in\Lambda_2}f(a+n_1+n_2)\overline{f(a+n_1+n_2')f(a+n_1'+n_2)}f(a+n_1'+n_2')\geq\frac{\eta^4}{2}.
\end{equation}
Fix such an $a$ and we will work with condition \eqref{largeu2}. We
may assume that $f$ is supported on the set $a+\Lambda_1+\Lambda_2$,
whose size is bounded by $|\Lambda_1|(1+100dc_2)$. Consider
\begin{equation}\label{u2sum}
\sum_{n_1,n_1'\in\Z}\sum_{n_2,n_2'\in\Lambda_2}f(a+n_1+n_2)\overline{f(a+n_1+n_2')f(a+n_1'+n_2)}f(a+n_1'+n_2').
\end{equation}
The contributions from $n_1,n_1'\in\Lambda_1$ are at least $\eta^4
|\Lambda_1|^2|\Lambda_2|^2/2$ by \eqref{largeu2}. If
$n_1,n_1'\notin\Lambda_1$, then we must have
$n_1,n_1'\in\Lambda_1+2\Lambda_2$ in order for the summand not to
vanish. There are at most $(200dc_2)^2|\Lambda_1|^2$ such pairs
$(n_1,n_1')$. Hence \eqref{u2sum} is at least
\[
[\eta^4/2-(200dc_2)^2]|\Lambda_1|^2|\Lambda_2|^2\geq\frac{\eta^4}{4}|\Lambda_1|^2|\Lambda_2|^2,
\]
provided that $c_2\leq\eta^2/400d$.

On the other hand, we claim that \eqref{u2sum} is equal to
\[ \int_0^1\int_0^1
|\hat{\Lambda}_2(x)|^2|\hat{f}(y)|^2|\hat{f}(x+y)|^2dxdy. \] Indeed,
the integral above can be written as
\[
\sum_{n_2,n_2'\in\Lambda_2}\sum_{r,s,t,u\in\Z}f(r)\overline{f(s)}f(t)\overline{f(u)}\int_0^1\int_0^1e((n_2-n_2')x+(r-s)y+(t-u)(x+y))dxdy.
\]
By orthogonality the integral above vanishes unless $u-t=n_2-n_2'$
and $r+t=s+u$. After making the change of variables $n_1=r-a-n_2$
and $n_1'=u-a-n_2$, it is evident that the expression above equals
\eqref{u2sum}.

It then follows that
\begin{align*}
\frac{\eta^4}{4}|\Lambda_1|^2|\Lambda_2|^2 &\leq\int_0^1\int_0^1
|\hat{\Lambda}_2(x)|^2|\hat{f}(y)|^2|\hat{f}(x+y)|^2dxdy \\
&\leq \int_0^1|\hat{f}(y)|^2\cdot\sup_y
\int_0^1|\hat{\Lambda}_2(x)|^2|\hat{f}(x+y)|^2dx \\
&\leq |\Lambda_1|(1+100dc_2)\sup_y
\int_0^1|\hat{\Lambda}_2(x)|^2|\hat{f}(x+y)|^2dx.
\end{align*}
Hence there exists $y=y(a)\in\R$ such that
\begin{equation}\label{largef}
\int_0^1|\hat{\Lambda}_2(x)|^2|\hat{f}(x+y)|^2dx\geq\frac{\eta^4}{5}|\Lambda_1||\Lambda_2|^2.
\end{equation}
If we write $f_y$ for the function $f_y(a)=f(a)e(ay)$, then the left
side above is
\begin{equation}\label{largef2} \|\hat{\Lambda}_2\cdot\hat{f}_y\|_2^2=\|\widehat{\Lambda_2*f_y}\|_2^2=\sum_{n\in\Z}|(\Lambda_2*f_y)(n)|^2=\sum_{n\in a+\Lambda_1+2\Lambda_2}|(\Lambda_2*f_y)(n)|^2.
\end{equation}
Observe that
\[ |(\Lambda_2*f_y)(n)|=\left|\sum_{n_2\in\Lambda_2}f_y(n+n_2)\right|=\left|\sum_{n_2\in\Lambda_2}f(n+n_2)e(n_2y)\right|. \]
Combining this with \eqref{largef} and \eqref{largef2} we get
\[ \E_{n\in a+\Lambda_1+2\Lambda_2}|\E_{n_2\in\Lambda_2}f(n+n_2)e(n_2y)|^2\geq \frac{\eta^4}{5}\cdot\frac{|\Lambda_1|}{|\Lambda_1+2\Lambda_2|} \geq\frac{\eta^4}{10}. \]
Recall that the above is true for at least $\eta^4|\Lambda|/2$
values of $a\in\Lambda$. We have thus shown that
\begin{equation}\label{largef3}
\E_{a\in\Lambda}\E_{n\in
a+\Lambda_1+2\Lambda_2}\sup_{y\in\R}\left|\E_{n_2\in\Lambda_2}f(n+n_2)e(n_2y)\right|^2\geq\frac{\eta^{8}}{20}.
\end{equation}
After changing the order of summation over $a$ and $n$ the left side
above is at most
\[ \E_{n\in\Lambda}\sup_{y\in\R}\left|\E_{n_2\in\Lambda_2}f(n+n_2)e(n_2y)\right|^2+300dc_1. \]
The proof is completed by combining this with \eqref{largef3}.
\end{proof}

\begin{remark}\label{rem:locdefn}
In the proof above it is crucial that the variables $n_2,n_2'$ are
restricted to an even narrower Bohr set than the variables
$n_1,n_1'$. One might think that a more natural way to define the
local Gowers $U^2$-norm is by restricting $n_1,n_2,n_1',n_2'$ to the
same Bohr set. An inverse theorem for this definition is obtained by
Green-Tao \cite{personal}, but with a significantly more involved argument. We
will be content with our slightly less orthodox definition
\eqref{u2loc}.
\end{remark}


\section{A Local Generalized von-Neumann Theorem}\label{sec:gvn-loc}

Fix a positive integer $s\geq 2$. We will be looking for
$s$-configurations $\{n_i+n_j+a:1\leq i\leq j\leq s\}$, where
$a\in\Lambda$ and $n_i\in\Lambda_i$. Here
$\Lambda,\Lambda_1,\cdots,\Lambda_s$ are all regular Bohr sets of
dimension $d$, and $\Lambda_1=c_1\Lambda$,
$\Lambda_i=c_i\Lambda_{i-1}$ for $2\leq i\leq s$, where
$0<c_1,\cdots,c_s<1/100d$. To this end, define the counting function
\[
T_s(\mathcal{F};\Lambda,\Lambda_1,\cdots,\Lambda_s)=\E_{a\in\Lambda}\E_{n_1\in\Lambda_1}\cdots\E_{n_s\in\Lambda_s}\prod_{1\leq
i\leq j\leq s}f_{ij}(n_i+n_j+a), \] where $\mathcal{F}=\{f_{ij}\}$
is a collection of $1$-bounded functions.

\begin{proposition}[Generalized von-Neumann, local version]\label{u2control-loc} For any $1\leq
i<j\leq s$ we have
$|T_s(\mathcal{F};\Lambda_1,\cdots,\Lambda_s)|\leq\|f_{ij}\|_{U^2(\Lambda,\Lambda_i,\Lambda_j)}$.
\end{proposition}

If one wants to control $T_s(\mathcal{F})$ by $f_{ii}$, it seems
that a different definition of the local $U^2$-norm is needed (see Remark \ref{rem:locdefn}.
Fortunately it is enough to have the result for $i\neq j$.

\begin{proof}
For notational convenience assume that $(i,j)=(1,2)$; the other
cases are treated in the same way. We write
\[
T_s(\mathcal{F})=\E_{a\in\Lambda}\E_{n_2\in\Lambda_2}\cdots\E_{n_s\in\Lambda_s}\prod_{2\leq
i\leq j\leq
s}f_{ij}(n_i+n_j+a)\E_{n_1\in\Lambda_1}\prod_{j=1}^sf_{1j}(n_1+n_j+a).
\]
By Cauchy-Schwarz,
\[
|T_s(\mathcal{F})|^2\leq\E_{a\in\Lambda}\E_{n_2\in\Lambda_2}\cdots\E_{n_s\in\Lambda_s}\left|\E_{n_1\in\Lambda_1}\prod_{j=1}^sf_{1j}(n_1+n_j+a)\right|^2.
\]
Expanding the square and rearranging the order of summation we get
\[ |T_s(\mathcal{F})|^2\leq
\E_{a\in\Lambda}\E_{n_1,n_1'\in\Lambda_1}\left|\E_{n_2\in\Lambda_2}f_{12}(n_1+n_2+a)\overline{f_{12}(n_1'+n_2+a)}\right|.
\]
Apply Cauchy-Schwarz again to get
\[
|T_s(\mathcal{F})|^4\leq\E_{a\in\Lambda}\E_{n_1,n_1'\in\Lambda'}\left|\E_{n_2\in\Lambda'}f_{12}(n_1+n_2+a)\overline{f_{12}(n_1'+n_2+a)}\right|^2.
\]
The right side above, after expanding out the square, is exactly
$\|f\|_{U^2(\Lambda,\Lambda_1,\Lambda_2)}^4$.
\end{proof}

For a Bohr set $\Lambda$, we write $2\cdot\Lambda$ for the set $\{2n:n\in\Lambda\}$. Note that this is contained in, but usually much smaller than the dilated set $2\Lambda$.

\begin{corollary}[Dichotomy, local version]\label{dich-loc}
Let $s\geq 2$ be a positive integer. Let $\Lambda$ be a regular Bohr
set of dimension $d$ and $A\subset\Lambda$ be a subset with
$|A|=\delta |\Lambda|$. Let $0<c_1,\cdots,c_s<1/100d$ be real
parameters with $c_1\leq\delta^s/3200ds^2$. Set
$\Lambda_1=c_1\Lambda$, $\Lambda_i=c_i\Lambda_{i-1}$ for $2\leq
i\leq s$. Suppose that $A$ does not contain any nontrivial
$s$-configurations. Then either one of the following statements
holds:
\begin{enumerate}
\item $|\Lambda_s|\leq 32s^2\delta^{-\binom{s+1}{2}}$;
\item The density of $A$ on $a+2\cdot\Lambda_i$ is at least
$(1+1/8s^2)\delta$ for some $a\in\Lambda$ and $1\leq i\leq s$ with
$a+2\cdot\Lambda_i\subset\Lambda$;
\item $\|1_A-\delta
1_{\Lambda}\|_{U^2(\Lambda,\Lambda_i,\Lambda_j)}\geq\delta^{\binom{s+1}{2}}/32s^2$
for some $1\leq i<j\leq s$.
\end{enumerate}
\end{corollary}

\begin{proof}
Assume that case (1) fails. Since $A$ does not contain any
non-trivial $d$-configurations, we have
\begin{equation}\label{noc}
T_s(1_A)\leq\frac{s^2}{|\Lambda_s|}\leq\frac{1}{32}\delta^{\binom{s+1}{2}},
\end{equation}
On the other hand, we have the decomposition into
$L+1=\binom{s}{2}+1$ components
\[
T_s(1_A)=T_s(\mathcal{F}_1)+\cdots+T_s(\mathcal{F}_L)+T_s(\mathcal{G}),
\]
such that each $\mathcal{F}_k$ has $1_A-\delta 1_{\Lambda}$ as the
$ij$-component function for some $i<j$, and $\mathcal{G}=(g_{ij})$
satisfies $g_{ij}=\delta 1_{\Lambda}$ for all $i<j$ and
$g_{ii}=1_A$. Using Proposition \ref{u2control-loc} we get
\begin{equation}\label{lu2}
T_s(1_A)\geq T_s(\mathcal{G})-s^2\|1_A-\delta
1_{\Lambda}\|_{U^2(\Lambda,\Lambda_i,\Lambda_j)} \end{equation} for
some $1\leq i<j\leq s$. To estimate $T_s(\mathcal{G})$, note that if
$a\in (1-2c_1)\Lambda$ and $n_i,n_j\in\Lambda_1$, then
$a+n_i+n_j\in\Lambda$. Hence by the regularity of $\Lambda$,
\[
T_s(\mathcal{G})\geq\delta^{\binom{s}{2}}\left[\E_{a\in\Lambda}\prod_{i=1}^s\left(\E_{n\in\Lambda_i}1_A(2n_i+a)\right)-200dc_1\right].
\]
Now write $\delta_i(a)$ for the density of $A$ on
$a+2\cdot\Lambda_i$, and let $E_i$ be the set of $a\in
(1-2c_1)\Lambda$ with $\delta_i(a)<(1-1/s)\delta$. Assume that case
(2) of the conclusion fails, so that $\delta_i(a)\leq
(1+1/8s^2)\delta$ for all $a\in (1-2c_1)\Lambda$. Since
\[ \E_{a\in\Lambda}\delta_i(a)=\E_{n\in\Lambda_i}\E_{a\in\Lambda}1_A(2n+a)\geq\delta-200dc_1\geq (1-1/4s^2)\delta, \]
we have
\[ (1-1/4s^2)\delta\leq
\frac{|E_i|}{|\Lambda|}(1-1/s)\delta+\left(1-\frac{|E_i|}{|\Lambda|}\right)(1+1/8s^2)\delta+200dc_1.
\]
From this it follows that $|E_i|\leq |\Lambda|/2s$. Let $E$ be the
union $E_1\cup\cdots E_s$. Then $|E|\leq |\Lambda|/2$. Hence
\[
T_s(\mathcal{G})\geq\delta^{\binom{s}{2}}\left(\frac{1}{2}\E_{a\in\Lambda\setminus
E}\delta_1(a)\cdots\delta_d(a)-200dc_1\right)\geq\delta^{\binom{s}{2}}\left(\frac{1}{8}\delta^s-200dc_1\right)
\geq\frac{1}{16}\delta^{\binom{s+1}{2}}.
\]
Combining this with \eqref{noc} and \eqref{lu2} we get case (iii) of
the conclusion.

\end{proof}

Compared with the usual dichotomy in Roth's argument, the second
case above is new. However it is certainly harmless as it already
implies a density increment.


\section{Obtaining Density Increment in the Local
Setting}\label{sec:proof2}

The iterative procedure of attacking Theorem \ref{thm2} is the
following. Suppose that $A$ does not contain a nontrivial
$s$-configuration. Then Corollary \ref{dich-loc} implies that there
is a large local Gowers $U^2$-norm. By Theorem \ref{u2inverse-loc}
we thus have a large Fourier coefficient (in the average sense).
Finally, this large Fourier coefficient leads to a density increment
on a smaller Bohr set. This final step of the argument is recorded
in the following lemma.

\begin{lemma}[Large Fourier coefficient leads to density increment, local version]\label{incr-fourier-loc}
Let $0<\eta<1$ be a parameter. Let
$\Lambda=\Lambda_{\theta,\epsilon,M}$, $\Lambda_1=c_1\Lambda$ be
regular Bohr sets of dimension $d$, where $c_1\leq
2^{-15}\eta^3d^{-1}$. Suppose that $\E_{a\in\Lambda}f(a)=0$. If
\[
\E_{a\in\Lambda}\sup_{y\in\R}\left|\E_{n_1\in\Lambda_1}f(a+n_1)e(n_1y)\right|^2\geq\eta^2,
\]
then either one of the following statements holds:
\begin{enumerate}
\item $f$ has density increment on a translate of $\Lambda_1$: $\E_{n_1\in\Lambda_1}f(a+n_1)\geq\eta^3/128$ for some
$a$ with $a+\Lambda_1\subset\Lambda$;
\item for any positive $c'\leq 2^{-13}\eta d^{-1}$, $f$ has density increment on some translate of a $(d+1)$-dimensional Bohr set $\Lambda'=\Lambda_{\theta',\epsilon',M'}$ with $\epsilon'=c'c_1\epsilon$ and $M'=c'c_1M$:
\[ \E_{n'\in\Lambda'}f(a+n')\geq\eta/16 \]
for some $a$ with $a+\Lambda'\subset\Lambda$.
\end{enumerate}
\end{lemma}

\begin{proof}
For $a\in\Lambda$ write $\delta(a)=\E_{n_1\in\Lambda_1}f(a+n_1)$.
Suppose that case (i) fails, so that $\delta(a)\leq\eta^3/128$ for
each $a\in (1-c_1)\Lambda$. Let $E$ be the set of $a\in
(1-c_1)\Lambda$ with $\delta(a)\leq -\eta/32$. Note that
\[
\E_{a\in\Lambda}\delta(a)=\E_{a\in\Lambda}\E_{n_1\in\Lambda_1}f(a+n_1)\geq
-200c_1d\geq -\eta^3/128 \] since $\E_{a\in\Lambda}f(a)=0$. This
gives the upper bound $|E|\leq 3\eta^2|\Lambda|/4$. It follows that
there exists $a\in (1-c_1)\Lambda\setminus E$ and $y\in\R$ such that
\[ \left|\E_{n_1\in\Lambda_1}f(a+n_1)e(n_1y)\right|\geq\eta/2. \]

Fix such an $a$ and $y$. Define $\Lambda'$ by taking
$\theta'=(\theta,y)\in\R^{d+1}$. Then for any $n'\in\Lambda'$,
\[
|\E_{n_1\in\Lambda_1}f(a+n_1+n')e((n_1+n')y)-\E_{n_1\in\Lambda_1}f(a+n_1)e(n_1y)|\leq
200c'd. \] After averaging over $n'$ and changing the order of
summation we get
\[
\left|\E_{n_1\in\Lambda_1}\E_{n'\in\Lambda'}f(a+n_1+n')e((a+n_1+n')y)\right|\geq\E_{n_1\in\Lambda_1}f(a+n_1)e(n_1y)-200c'd\geq\frac{\eta}{4}.
\]
On the other hand, for $n'\in\Lambda'$ we have $|1-e(yn')|\leq
8\epsilon'$. Hence
\begin{align*}
\left|\E_{n_1\in\Lambda_1}\E_{n'\in\Lambda'}f(a+n_1+n')e((a+n_1+n')y)\right|
&\leq\E_{n_1\in\Lambda_1}\left|\E_{n'\in\Lambda'}f(a+n_1+n')e(n'y)\right|
\\
&\leq\E_{n_1\in\Lambda_1}\left|\E_{n'\in\Lambda'}f(a+n_1+n')\right|+8\epsilon'.
\end{align*}
It then follows that
\[
\E_{n_1\in\Lambda_1}\left|\E_{n'\in\Lambda'}f(a+n_1+n')\right|\geq\frac{\eta}{8}.
\]
Note that
\[ \E_{n_1\in\Lambda_1}\E_{n'\in\Lambda'}f(a+n_1+n')\geq\delta(a)-200c'd\geq -\frac{\eta}{32}-200c'd\geq -\frac{\eta}{16}\]
since $a\notin E$. Hence there exists $n_1\in\Lambda_1$ such that
\[ \E_{n'\in\Lambda'}f(a+n_1+n')\geq\frac{\eta}{16}. \]
\end{proof}

\begin{proposition}[Iterative step for Theorem
\ref{thm2}]\label{prop2}
Let $\Lambda=\Lambda_{\theta,\epsilon,M}$
be a regular Bohr set of dimension $d$ and let $A\subset\Lambda$ be
a subset with $|A|=\delta|\Lambda|$. Set
\[ x_1=2^{-85}s^{-24}d^{-1}\delta^{6s(s+1)},\ \ x_2=\cdots=x_s=2^{-20}s^{-4}d^{-1}\delta^{s(s+1)}. \]
Then there exists $c_i\in [x_i,2x_i]$ such that the Bohr sets
$\Lambda_1=c_1\Lambda$ and $\Lambda_i=c_i\Lambda_{i-1}$ ($2\leq
i\leq s$) are all regular. Moreover, either one of the following
statements holds:
\begin{enumerate}
\item $A$ contains a non-trivial $s$-configuration;
\item $|\Lambda_s|\leq 32s^2\delta^{-\binom{s+1}{2}}$;
\item the density of $A$ on $a+2\cdot\Lambda_i$ is at least
$(1+1/8s^2)\delta$ for some $a\in\Lambda$ and $1\leq i\leq s$;
\item the density of $A$ on $a+\Lambda_i$ is at least
$\delta+2^{-54}s^{-16}\delta^{4s(s+1)}$ for some $a\in\Lambda$ and
$1\leq i\leq s$;
\item there is a regular Bohr set $\Lambda'=\Lambda_{\theta',\epsilon',M'}$ of dimension $d+1$, where $\epsilon'\geq c'c_1\cdots c_s\epsilon$, $M'\geq c'c_1\cdots c_sM$
with
\[ c'\geq 2^{-37}s^{-8}d^{-1}\delta^{2s(s+1)}, \]
such that the density of $A$ on $a+\Lambda'$ is at least
$\delta+2^{-28}s^{-8}\delta^{2s(s+1)}$ for some $a$.
\end{enumerate}
\end{proposition}

\begin{proof}
Assume that case (1) fails. By Lemma \ref{findreg} we may choose
$c_i\in [x_i,2x_i]$ ($1\leq i\leq s$) such that
$\Lambda_1=c_1\Lambda,\Lambda_2=c_2\Lambda_1,\cdots$ are all regular
Bohr sets. By Corollary \ref{dich-loc}, either we are in case (2) or
(3), or
\[ \|1_A-\delta 1_{\Lambda}\|_{U^2(\Lambda,\Lambda_i,\Lambda_j)}\geq\delta^{\binom{s+1}{2}}/32s^2 \]
for some $1\leq i<j\leq s$. By Theorem \ref{u2inverse-loc},
\[ \E_{a\in\Lambda}\sup_{y\in\R}|\E_{n_j\in\Lambda_j}f(a+n_j)e(n_jy)|^2\geq 2^{-46}s^{-16}\delta^{4s(s+1)}. \]
Finally apply Proposition \ref{incr-fourier-loc} to arrive at case
(4) or (5).
\end{proof}

\begin{proof}[Proof of Theorem \ref{thm2}]
Set
\[ d_0=1,\ \ \epsilon_0=1,\ \ M_0=N,\ \,\delta_0=\delta/2. \]
Start with the $d_0$-dimensional regular Bohr set
$\Lambda_0=\Lambda_{1,\epsilon_0,M_0}=[-N,N]$. Let $A_0=A$ so that
the density of $A_0$ in $\Lambda_0$ is $\delta_0$. We will apply
Proposition \ref{prop2} repeatedly to obtain sequences of Bohr sets
$\Lambda_0,\Lambda_1\cdots,\Lambda_k$ and subsets
$A_1\subset\Lambda_1,\cdots,A_k\subset\Lambda_k$ of densities
$\delta_1,\cdots,\delta_k$ satisfying the following properties:
\begin{enumerate}
\item (dimension bound) $d_i\leq d_{i-1}+1$ for each
$i=1,2,\cdots,k$;
\item (density increment) for each $i=1,2,\cdots,k$ we have
\begin{equation}\label{di1}
\delta_i\geq\delta_{i-1}+2^{-54}s^{-16}\delta_{i-1}^{4s(s+1)};
\end{equation}
moreover if $d_i=d_{i-1}+1$ then
\begin{equation}\label{di2}
\delta_i\geq\delta_{i-1}+2^{-28}s^{-8}\delta_{i-1}^{2s(s+1)};
\end{equation}
\item (size bound) $\epsilon_i\geq c\epsilon_{i-1}$ and $M_i\geq
cM_{i-1}$ for each $i=1,2,\cdots,k$, where
\begin{equation}\label{shrink}
 c=s^{-100s}d_k^{-s}\delta^{10s^3};
 \end{equation}
\item for each $i=1,2,\cdots,k$, existence of nontrivial
$s$-configurations in $A_i$ implies existence of nontrivial
$s$-configurations in $A_{i-1}$;
\item (end of iterations) either $A_k$ contains a nontrivial
$s$-configuration, or
\begin{equation}\label{small}
|c\Lambda_k|\leq 32s^2\delta^{-\binom{s+1}{2}},
\end{equation}
where the constant $c$ is defined in \eqref{shrink}.
\end{enumerate}

It remains to show that \eqref{small} does not occur. This will
imply that $A_k$ contains a nontrivial $s$-configuration, and thus
so does the original set $A$. By \eqref{di1}, \eqref{di2}, and
\eqref{shrink}, we have
\[ k\leq 2^{55}s^{16}\delta^{-4s(s+1)}, \]
\[ d_k\leq 2^{29}s^8\delta^{-2s(s+1)}, \]
\[ \epsilon_k\geq c^k,\ \ M_k\geq c^kN. \]
By Lemma \ref{bohrlow} we have the lower bound
\[ |c\Lambda_k|\geq (c\epsilon_k)^{d_k}M_k\geq c^{(k+1)d_k}c^kN\geq N\exp(-2^{94}s^{27}\delta^{-6s(s+1)}\log(1/\delta)). \]
The inequality
\[ N\exp(-2^{94}s^{27}\delta^{-6s(s+1)}\log(1/\delta))>
32s^2\delta^{-\binom{s+1}{2}} \] can be readily verified under the
assumption on $\delta$. This shows that \eqref{small} cannot happen.
\end{proof}


\section{Proof of Corollary \ref{sumfree}}

The general strategy of proving Corollary \ref{sumfree} is the same
as in \cite{sumfree} (see also \cite{dousse}). It follows from the following proposition, just as
Theorem 1.1 in \cite{sumfree} follows from Theorem 1.2 there.

\begin{proposition}\label{sumfreeh}
Let $h$ be a sufficiently large positive integer. Let $X,Y$ be two
finite subsets of positive integers with
\begin{equation}\label{eqh}
 h^{-29}|Y|\geq |X|\geq
\exp(\exp(Ch^2\log h)) \end{equation}
for some sufficiently large
constant $C>0$. Then $Y$ contains a subset $Z$ of size $|Z|=h$
disjoint from $X$ and is sumfree with respect to the union $X\cup
Y$. In other words, $z_1+z_2\notin X\cup Y$ for distinct $z_1,z_2\in
Z$.
\end{proposition}

Using the notation in \cite{sumfree}, we could take
$F(h)=\exp(\exp(Ch^2\log h))$, and thus Corollary \ref{sumfree} is
true with the lower bound $|B|\geq g(n)\log n$, where $g(n)=cm/\log
m$. Here $c$ is a small constant and $F(m)=n^{1/2}$. Thus $m$ can be
taken to be $(\log\log n)^{1/2-o(1)}$. We now focus on proving
Proposition \ref{sumfreeh}.

\begin{proposition}\label{sumfreep}
Let $h,X,Y$ be as in the statement of Theorem \ref{sumfreeh}.
Suppose that $Y$ does not contain any subset $Z$ of size $h$ which
is sumfree with respect to $X\cup Y$. Then there is a subset
$Y_3\subset Y$ satisfying the following properties:
\begin{enumerate}
\item (large size) $|Y_3|\gg h^{-29}|Y|$;
\item (small doubling) $|Y_3+Y_3|\ll h^{181}|Y_3|$;
\item the set $2\cdot Y_3$ is disjoint from $X\cup Y$. In other words,
$2y_3\notin X\cup Y$ for any $y_3\in Y_3$.
\end{enumerate}
\end{proposition}

\begin{proof}
See Section 6 in \cite{sumfree}.
\end{proof}

To deduce Proposition \ref{sumfreeh} from Proposition \ref{sumfreep}, it
suffices to find an $h$-configuration in $Y_3$. The small doubling
property of $Y_3$ allows us to extract a subset of $Y_3$ lying
densely inside an interval. This is achieved in \cite{sumfree} by an
application of Freiman's theorem. A more economical way of doing
this is to use an embedding lemma due to Ruzsa \cite{ruzsa1,ruzsa2}
(see also Lemma 5.26 in \cite{tao-vu}).

\begin{lemma}[Ruzsa's embedding lemma]\label{embedding}
Let $A\subset\Z$ be a finite set with $|A-A|\leq K|A|$. Then there
is a subset $A'\subset A$ with $|A'|\geq |A|/2$ such that $A'$ is
Freiman isomorphic to a subset $A''\subset\Z/N\Z$ with $N\ll K|A|$. In
other words, there is a bijection $\phi:A'\rightarrow A''$ such that
\[ a_1'+a_2'=a_3'+a_4'\Longleftrightarrow
\phi(a_1')+\phi(a_2')=\phi(a_3')+\phi(a_4') \] whenever
$a_1',a_2',a_3',a_4'\in A'$.
\end{lemma}

\begin{proof}[Proof of Proposition \ref{sumfreeh}]
As noted above, it suffices to show that $Y_3$ contains a nontrivial
$h$-configuration. Since $|Y_3+Y_3|\ll h^{181}|Y_3|$, we have
$|Y_3-Y_3|\ll h^{362}|Y_3|$ (this is a consequence of the Ruzsa
triangle inequality; see Section 2.3 in \cite{tao-vu}). By Ruzsa's
embedding lemma, there is a subset $Y_4\subset Y_3$ with $|Y_4|\geq
|Y_3|/2$ that is Freiman isomorphic to a subset $Y_4'\subset\Z/N\Z$
with $N\ll h^{362}|Y_3|$.

Identify $\Z/N\Z$ with $[N]$ and view $Y_4'$ as a subset of the
integers. To find a nontrivial $h$-configuration in $Y_3$, it
suffices to find a nontrivial $h$-configuration in $Y_4'$. The
density of $Y_4'$ in $[N]$ is
\[ \delta=\frac{|Y_4'|}{N}\gg h^{-362}. \]
By Theorem \ref{thm2}, it thus suffices to show that
\[ h^{-362}\gg (\log N)^{-1/20h^2}. \]
Simple algebra reveals that this is implied by the condition
\eqref{eqh}.
\end{proof}

\bibliography{linear_patterns}{}
\bibliographystyle{plain}

\end{document}